\documentclass{article}
\usepackage{graphicx} 
\usepackage{amsmath}
\usepackage{amsthm}
\usepackage{amsfonts}
\usepackage{amssymb}
\newtheorem{theorem}{Teorema}
\newtheorem{lemma}[theorem]{Lemma}
\newtheorem{proposition}[theorem]{Proposition}
\newtheorem{corollary}[theorem]{Corollary}
\newtheorem{definition}{Definición}
\newtheorem{example}{Example}

\title{Derivation over twisted group ring and its applications}
\author{Alvaro Oteor Sanchez, aos073@ual.es}
\date{June 2025}

\allowdisplaybreaks
\begin{document}

\maketitle
\section{Introduction}
In algebra it is usual to use two existing mathematicals objects to produce a new one with new properties. The most notable example are groups rings, where a ring and a group can be used to construct a new ring. One particular properties is that even if the original ring is commutative, the group ring based on that ring does not have to be necessary commutative. 

The origin of this concept can be found in the work of T. Molen \cite{Molien1897}, and G. Frobenius \cite{Frobenius1896}, \cite{Frobenius1896b}. Later, I. Schur presented the so called Schur lemma \cite{Schur1905}, a foundational result in representation theory which can be understood in terms of groups ring. 

The development of group rings as an independent subject can be tracked to \cite{Maschke1899}, where Mashkle proved necessary conditions for a group ring to be semisimple. However, it is in the work of  E. Noether and R. Brauer \cite{Noether1929}, \cite{Brauer1929} where they passed to have a central role in non commutative algebra, seen as an algebraic object by its own. 

Twisted group rings are a generalization of group rings, where the product is still induced by the product of the group, but where the result is twisted by a scalar which is determined by a 2-cocycle. They appear close to the beginning of group ring, but it was not until \cite{HandelmanLawrenceSchelter1978} that they had its own theory. Since then, they play a central role in several filds, from lie theory \cite{McConnell1975} to Heisenberg categorification \cite{RossoSavage2017}. 

One application of group ring is in coding theory, with the so called group ring codes. In \cite{HurleyHurley2009} the authors proposed a method to produce linear codes from group rings. However, a unified theory as well as some of their most important properties can be found in \cite{Hughes2001}. Recently, those ideas has been generalized to be used in twisted group ring, such as in \cite{DelaCruz2021}, where it is proven that a wide range of codes can be seen as a particular instance of twisted group rings. 

Another line of research to obtain new generalization of group ring has been thougt it derivations, as can be seen in \cite{CREEDON2019247}. In this paper the authors study the derivation of a group ring, and use it to determine the derivations of a group ring of a finite field and the dihedral group. However, the authors recognize that there is a leak of scientific research about derivations over group rings with positive characteristic. Nevertheless, there have been important advances such as the work of Ferrero, Giambruno and Polcino Milies \cite{Ferrero95} where they proved what we will call the FGP theorem.

\begin{theorem}[{\cite[Theorem 1.1]{Ferrero95}}]
Let \( R \) be a semiprime ring and \( G \) a torsion group such that
\[
[G : Z(G)] < \infty,
\]
where \( Z(G) \) denotes the center of \( G \). Suppose that either \(\operatorname{char} R = 0\) or, for every characteristic \( p \) of \( R \),
\[
p \nmid o(g), \quad \text{for all } g \in G.
\]
Then every \( R \)-derivation of \( RG \) is inner.
\end{theorem}

Derivarions play a crucial role in Hochschild cohomology, in particular, the first Hochschild cohomology group of a ring $B$, $HH^1(B)$ has information about the relations between derivations and inner derivations of a particular group. In \cite{Fleischmann1993}, the authors results implies that the $HH^1(kG)$ is nontrivial for any finite group $G$ of order divisible by a prime $p$, where $k$ is an
algebraically closed field of characteristic $p$. Recently, \cite{TODEA2023107192} prove some results regarding when the first Hochschild cohomology group does not vanish for some twisted group ring. However, their results do not apply to commutative groups nor the dihedral group.

In this paper, we will generalize the results of \cite{Ferrero95} to certain cases of twisted group rings, as well as generalize the work of \cite{CREEDON2019247} to the case of twisted group ring over a finite field. Finally, we use the results to the case of the dihedral group to study the first Hochschild cohomology group.

\section{Derivations of twisted group ring}
In this section we will prove our main results about derivations over twisted group ring. To do so, we will start with some basic defitions.

Now we will recall the concept of twisted group ring.

\begin{definition}
    Let $G$ be a group and $A$ an abelian group. A $2-$cocycle is a map $\alpha: G\times G \longrightarrow A$ such that
    \begin{equation}
        \alpha(x,y) \alpha(xy,z) = \alpha(y,z) \alpha(x,yz)
    \end{equation}
    for all $x,y,z \in G$. We will say that the a $2-$cocycle $\alpha$ is normaliced if $\alpha(1_G,g) = \alpha(g,1_G) = 1_A$ for all $g\in G$.
The set of all 2-cocycles of $G$ is denoted by $Z^2(G, A)$.
\end{definition}

To understand the structure of twisted group rings we need to introduce the second cohomology group and the coboundaries

\begin{definition}
Let $\varphi: G \to A$ be a map. The coboundary $\partial \varphi: G \times G \to A$ is defined by
\[
\partial \varphi(x,y) := \varphi(x) \, \varphi(y) \, \varphi(xy)^{-1}, \quad \forall x,y \in G.
\]
Then $\partial \varphi$ is a 2-cocycle. The set of all coboundaries is denoted by $B^2(G,A) \subseteq Z^2(G,A)$.
\end{definition}

Now we can recall the second cohomology group
\begin{definition}
The second cohomology group of $G$ with values in $A$ is the factor group
\[
H^2(G,A) := Z^2(G,A) / B^2(G,A) = \{ [\alpha] : \alpha \in Z^2(G,A) \},
\]
where $[\alpha] = \alpha B^2(G,A)$ denotes the coset of $\alpha$ modulo coboundaries.
\end{definition}

Two cocycles in the same equivalence class are called cohomologous.

\begin{definition}
Two 2-cocycles $\alpha, \beta \in Z^2(G,A)$ are said to be cohomologous if there exists a map 
$\varphi: G \to A$ such that
\[
\beta(x,y) = \alpha(x,y) \, \partial \varphi(x,y), \quad \forall x,y \in G.
\]
Equivalently, $\alpha$ and $\beta$ belong to the same coset in $H^2(G,A)$.
\end{definition}

Now we can recall the concept of twisted group rings

\begin{definition}
    Let $G$ be a group, $R$ a ring and $\alpha:G\times G\longrightarrow J\subset U(R)\cap Z(R)$ a  $2-$cocycle with $(J,\cdot)$ a group. Then the twisted group ring $R^\alpha G$ is the $R$-vector space with base $\{ \overline{g}; g\in G\}$, and product given by 
\begin{equation}
    (a\overline{g}) (b \overline{ h} ) = ab \alpha(g,h) \overline{gh}
\end{equation}
and extended by linearity. 
\end{definition}

It is a well know fact that this structure is associative \cite{Passman1989}.
\begin{proposition} 
    The twisted group ring is a associative
\end{proposition}

In addition, it is known \cite{karpilovsky1987algebraic} that two cohomologous $2-$cocycles $\alpha, \beta : G \times G \longrightarrow J\subset U(R)\cap Z(R)$ lead to isomorphic twisted group rings $F^\alpha G \equiv F^\beta G$ and vice versa. Also, there is always a normalized $2-$cocycle in each class of equivalence, so we can assume our all $2-$cocycle of this work are normalized. 

Now, we give a brief introduction to the derivations.

\begin{definition}
A derivation of a ring \( R \) is a mapping \( d : R \to R \) satisfying
\begin{align}
    d(a + b) &= d(a) + d(b), && \text{for all } a, b \in R,  \\
    d(ab) &= d(a)b + a d(b), && \text{for all } a, b \in R. 
\end{align}
The second equation is known as Leibniz’s rule. We will note \( \mathrm{Der}(R) \) the set of derivations of a ring \( R \).

We say $f\in \mathrm{Der}(R)$ is said to be inner if there exists $a\in R$ such that $f(x) = ax-xa$ for all $x\in R$, and the set of all inner derivations is noted by $\mathrm{Inn}(R)$.
\end{definition}

Note that if \( R \) is a unital ring, then \( d(1) = 0 \), since
\[
d(1) = d(1 \cdot 1) = d(1) \cdot 1 + 1 \cdot d(1) = 2d(1) \Rightarrow d(1) = 0.
\]
If we note by $Z(R)$ the center subring of $R$, then the set $\mathrm{Der}(R)$ can be endowed with a structure of $Z(R)-$module.

\begin{lemma}[Lemma 2.1, \cite{CREEDON2019247}] \label{Lemma 2.1}
Let \( d \) be a derivation of a ring \( R \). Then:

\begin{itemize}
    \item[(i)] For all \( a_1, a_2, \dots, a_m \in R \),
    \begin{equation}
        d\left( \prod_{i=1}^m a_i \right) = \sum_{i=1}^m \left( \prod_{j=1}^{i-1} a_j \right) d(a_i) \left( \prod_{j=i+1}^m a_j \right). \label{eq:product_rule}
    \end{equation}

    \item[(ii)] For all \( a \in R \) and \( m \in \mathbb{N} \),
    \begin{equation}
        d(a^m) = \sum_{i=0}^{m-1} a^i d(a) a^{m-1-i}. \label{eq:power_rule}
    \end{equation}

    \item[(iii)] For all units \( a \in R \) of order \( n \),
    \begin{equation}
        \sum_{i=0}^{n-1} a^i d(a) a^{n-1-i} = 0. \label{eq:unit_identity}
    \end{equation}

    \item[(iv)] For all \( a \in R \) that commute with \( d(a) \) and all \( k \in \mathbb{N} \),
    \begin{equation}
        d(a^k) = k a^{k-1} d(a). \label{eq:commuting_rule_natural}
    \end{equation}

    \item[(v)] For all units \( a \in R \) that commute with \( d(a) \) and all \( k \in \mathbb{Z} \),
    \begin{equation}
        d(a^k) = k a^{k-1} d(a). \label{eq:commuting_rule_integer}
    \end{equation}
\end{itemize}
\end{lemma}

We will focus our study in the case of $R-$derivations.

\begin{definition}
Let \( R \) be a subring of a ring \( S \). Then a derivation \( d : S \to S \) is an \emph{\( R \)-derivation} if \( d(R) = \{0\} \).
\end{definition}

In particular, let \( R^\alpha G \) be a group ring. Then a derivation \( d : R^\alpha G \to R^\alpha G \) is an \( R \)-derivation if \( d(R) = \{0\} \).

Let $S$ be a generator set of $G$. Now we will introduce necessary and sufficient conditions for a map $f : S \longrightarrow R^\alpha G$ to be extended to a $R-$derivation.
\begin{definition}
    Let $f:A\subset R \longrightarrow R$. An extension by derivation  $f^*:R \longrightarrow R$ is a derivation such that $f^*|_A = f$
\end{definition}
\begin{theorem}
    Let $G= \langle S \mid T \rangle$ be a group with generators $S$ and relations $T$. Let $R$ be a commutative ring and $\alpha: G \times G \longrightarrow U(R)$ a normalized $2-$cocycle, and let $f:S \longrightarrow K^{\alpha}G$. Then
    \begin{enumerate}
    \item If there exist an extension by derivation $f^*:K^{\alpha}G \longrightarrow K^{\alpha}G$ of $f$, then it is unique and given by 
        \[
f^* (w) =
\begin{cases}
f(w) & \text{if } w\in S, \\
- \frac{1}{\alpha(w,w^{-1})} \overline{w} f(w^{-1})  \overline{w} & \text{if } w \in S^{-1}, \\
0 & \text{if } w_i = 1 \\
\frac{1}{\beta_w}\sum_{i=1}^k \left( \left( \prod_{j=1}^{i-1}  \overline{w}_j \right) f^*(w_i) \left( \prod_{j=i+1}^{k}  \overline{w}_j \right) \right) & \text{if } w = \overline{\prod_{i=1}^k w_i }; w_i \in S \cup S^{-1} 
\end{cases}
\]
and extended by linearity, where 
    \begin{equation}
        \beta_w = \prod _{i=1}^{n-1}\alpha(w_i,\prod_{i<j}^{n} w_j)
    \end{equation}
    
    \item There exists an extension of $f$ if and only if $f^*$ satisfies $f^*(T)=0$. 
    \end{enumerate}
\end{theorem}
\begin{proof}
    1) We have that $f^*|_S = f$, so we only have to prove it is a derivation, and secondly we will prove that it is well defined. For the first one, we have to prove it for the elements of $G$ in $K^\alpha G$. Let $u = \overline{\prod_i^n w_i}$, $v = \overline{\prod_i^m h_i}$ $w_i,h_i \in S$.
    
    \begin{align*}
        f^*(\overline{u} \cdot \overline{v}) & = f^*( \alpha(u,v) \overline{uv}) \\
        & = \alpha(u,v) f^*( \overline{uv}) \\
        & = \alpha(u,v) \frac{1}{\beta_{uv}}\left( \sum_{i=1}^n \left( \left( \prod_{j=1}^{i-1}  \overline{w}_j \right) f^*(w_i) \left( \prod_{j=i+1}^{n}  \overline{w}_j  \prod_{j=1}^{m}  \overline{h}_j\right) \right)  \right) + \\
        & + \alpha(u,v) \frac{1}{\beta_{uv}} \left(\sum_{k=1}^m \left( \left( \prod_{j=1}^{n}  \overline{w}_j  \prod_{j=1}^{k}  \overline{h}_j\right) f^*(h_k)  \prod_{j=k+1}^{m}  \overline{h}_j  \right) \right)  \\
        & = \alpha(u,v) \frac{1}{\beta_v \beta_v \alpha(u,v)}\left( \sum_{i=1}^n \left( \left( \prod_{j=1}^{i-1}  \overline{w}_j \right) f^*(w_i) \left( \prod_{j=i+1}^{n}  \overline{w}_j  \prod_{j=1}^{m}  \overline{h}_j\right) \right) \right) +        \\
        & + \alpha(u,v) \frac{1}{\beta_v \beta_v \alpha(u,v)} \left( \sum_{k=1}^m \left( \left( \prod_{j=1}^{n}  \overline{w}_j  \prod_{j=1}^{k}  \overline{h}_j\right) f^*(h_k)  \prod_{j=k+1}^{m}  \overline{h}_j  \right) \right) \\
        & =  \frac{1}{\beta_v \beta_v}\left( \sum_{i=1}^n \left( \left( \prod_{j=1}^{i-1}  \overline{w}_j \right) f^*(w_i) \left( \prod_{j=i+1}^{n}  \overline{w}_j  \prod_{j=1}^{m}  \overline{h}_j\right) \right) \right) + \\
        & + \frac{1}{\beta_v \beta_v} \left(\sum_{k=1}^m \left( \left( \prod_{j=1}^{n}  \overline{w}_j  \prod_{j=1}^{k}  \overline{h}_j\right) f^*(h_k)  \prod_{j=k+1}^{m}  \overline{h}_j  \right) \right) \\
        & = \frac{1}{\beta_u} \sum_{i=1}^n \left( \left( \prod_{j=1}^{i-1}  \overline{w}_j \right) f^*(w_i)  \prod_{j=i+1}^{n}  \overline{w}_j     \right) \overline{v} + \\
        & + \frac{1}{\beta_v}\overline{u}\sum_{k=1}^m \left( \left(  \prod_{j=1}^{k}  \overline{h}_j\right) f^*(h_k)  \prod_{j=k+1}^{m}  \overline{h}_j  \right) \\
        & = f(\overline{u}) \overline{v} + \overline{u} f( \overline{v})
    \end{align*}
    Where we have used that $\beta_{uv} = \beta_{u} \beta_{v} \alpha(u,v)$.

    Now, suppose that there exits an extension by derivation $d$ of $f$. So  $d(w) = f(w)$ if $w\in S$. Then 
    \begin{equation}
        d(\overline{w} \overline{w^{-1}}) = \alpha(w,w^{-1}) d(\overline{1}) = 0
    \end{equation}
    and as $d$ is a derivation
    \begin{equation}
        d(\overline{w} \overline{w^{-1}}) = d(\overline{w}) \overline{w^{-1}} + \overline{w} d(\overline{w^{-1}})
    \end{equation}
    With both equations: 
    \begin{equation}
         d(\overline{w}) \overline{w^{-1}} + \overline{w} d(\overline{w^{-1}}) = 0
    \end{equation}
    Multiply by $\overline{w^{-1}}$ by the left
    \begin{equation}
         \overline{w^{-1}} d(\overline{w}) \overline{w^{-1}} + \alpha(w,w^{-1}) d(\overline{w^{-1}}) = 0
    \end{equation}
    and therefore
    \begin{equation}
        d(w^{-1}) = \frac{-1}{\alpha(w,w^{-1})}\overline{w^{-1}} d(\overline{w})\overline{w^{-1}} = f^*(w^{-1})
    \end{equation}
    finally, by \ref{Lemma 2.1}, if we take $a_i = \overline{w_i}$, then
    \begin{align*}
        d\left(\overline{\prod_{i=1}^n w_i}\right) =&  d\left( \frac{1}{\beta_{\prod_{i=1}^n w_i}} \prod_{i=1}^n \overline{w_i} \right) \\ 
        = & \frac{1}{\beta_{\prod_{i=1}^n w_i}} d\left(\prod_{i=1}^n \overline{w_i} \right) \\ 
        = & \frac{1}{\beta_{\prod_{i=1}^n w_i}} \sum_{k=1}^n \left(\prod_{i=1}^{k-1}\overline{w_i}\right) d(\overline{w_i}) \left(\prod _{k+1}^n \overline{w_i} \right) \\ 
        = & \frac{1}{\beta_{\prod_{i=1}^n w_i}} \sum_{k=1}^n \left(\prod_{i=1}^{k-1}\overline{w_i}\right) f(\overline{w_i}) \left(\prod _{k+1}^n \overline{w_i} \right) \\
        = & f^*\left(\overline{\prod_{i=1}^n w_i}\right)
    \end{align*}
    As we wanted to prove. 

    2) If $f^*$ is well defined, then by 1) existe una derivacion. But in this case, we have that for all $t\in T$ as an element of $G$, we have $\overline{t}= \overline{1}$, so $f^*(\overline{t}) = f^*(\overline{1}) = 0$. 
    
    For the other implication, if $\overline{a} = \overline{b}$, then there exist $t_1,t_2 \in T$ such that $a = t_1 b t_2$. Then
    \begin{align*}
        f^*(\overline{a}) = & f^* ( \overline{t_1 b t_2}) \\  
        = & \frac{1}{\alpha(t_1,bt_2)} \left( f^*(\overline{t_1})\overline{bt_2} + \overline{t_2} f^*(\overline{bt_2}) \right) \\  
        = & \frac{1}{\alpha(1,b)} \left( f^*(\overline{t_1})\overline{bt_2} + \overline{t_2} f^*(\overline{bt_2}) \right) \\
        = &  \left( f^*(\overline{t_1})\overline{bt_2} + \overline{t_2} f^*(\overline{bt_2}) \right) \\
        = & f^*(\overline{bt_2}) \\  
        = & \frac{1}{\alpha(b,t_2)} \left(f^*(\overline{b})\overline{t_2} + \overline{b} f^*(\overline{t_2}) \right)\\  
        = & \frac{1}{\alpha(b,1)} \left(f^*(\overline{b})\overline{t_2} + \overline{b} f^*(\overline{t_2}) \right)\\  
        = & f^*(\overline{b})\overline{t_2} + \overline{b} f^*(\overline{t_2}) \\  
        = & f^*(\overline{b}) 
    \end{align*}
\end{proof}

As the result of the previous theorem, we have 

\begin{theorem}
    Let \( G = \langle S \mid T \rangle \) be a group, where \( S \) is a generating set and \( T \) is a set of relations. Let $f:S \longrightarrow K^{\alpha}G$ and let $K$ be an algebraic extension of a prime field. Then:
    \begin{enumerate}
        \item If there exists a derivation extension $f^*:K^{\alpha}G \longrightarrow K^{\alpha}G$ of $f$, it is unique and is given by
        \[
f^* (w) =
\begin{cases}
f(w) & \text{if } w\in S, \\
- \frac{1}{\alpha(w,w^{-1})} \overline{w} f(w^{-1})  \overline{w} & \text{if } w \in S^{-1}, \\
0 & \text{if } w_i = 1, \\
\frac{1}{\beta_w}\sum_{i=1}^k \left( \left( \prod_{j=1}^{i-1}  \overline{w}_j \right) f^*(w_i) \left( \prod_{j=i+1}^{k}  \overline{w}_j \right) \right) & \text{if } w = \overline{\prod_{i=1}^k w_i }, \; w_i \in S \cup S^{-1}
\end{cases}
\]
        where 
        \begin{equation}
            \beta_w = \prod _{i=1}^{n-1}\alpha(w_i,\prod_{i<j}^{n} w_j).
        \end{equation}
        \item A derivation extension of $f$ exists if and only if the above $f^*$ satisfies $f^*(T)=0$.
    \end{enumerate}
\end{theorem}

\section{FGP theorem over twisted group ring}
In \cite{Ferrero95}, we have the following main theorem
\begin{theorem}[{\cite[Theorem 1.1]{Ferrero95}}]
Let \( R \) be a semiprime ring and \( G \) a torsion group such that
\[
[G : Z(G)] < \infty,
\]
where \( Z(G) \) denotes the center of \( G \). Suppose that either \(\operatorname{char} R = 0\) or, for every characteristic \( p \) of \( R \),
\[
p \nmid o(g), \quad \text{for all } g \in G.
\]
Then every \( R \)-derivation of \( RG \) is inner.
\end{theorem}
We will generalize that result to our case of twisted group algebras, which are twisted group ring where the original ring is also a $k$-algebra. To do so, we need some previous results.

Let $R$ be a $k-$algebra, $G$ a group and $\alpha: G\times G \longrightarrow k^*$ . Then we have that $Im(\alpha) \subset Z=Z(R)$, so $Z^\alpha G$ is well defined.

Then every $R$-derivation $d$ induces a $Z$-derivation of $Z^\alpha G$. In fact, given $r \in R$ and $g \in G$, we have $r  \overline{g} = \overline{g}  r$, so $r d(g) = d(g)  r$, which means that $d(g) \in Z^\alpha G$ as a $Z$-vector space. From this, we deduce that $d_Z$, the restriction of $d$ to $Z^\alpha G$, is well-defined.

Now, let $T$ be another ring such that $R \subseteq T$ and $Z(R) \subseteq Z(T)$. Then, identifying $T^\alpha G$ with $T \otimes_Z Z^\alpha G$, $d$ can be extended in a natural way to a $T$-derivation $d_T$ of $T^\alpha G$ by defining
\[
d_T = 1 \otimes d : T \otimes_Z Z ^\alpha G \longmapsto T \otimes_Z Z ^\alpha G.
\]

\begin{proposition}
Let $R \subseteq T$ be $k-$algebras, such that $Z(R) \subseteq Z(T)$, 
and let $d$ be an $R$-derivation of a group ring $R ^\alpha G$. Then $d$ is inner if and only if $d_T$ is inner.
\end{proposition}

\begin{proof}
    Assume first that $d$ is an inner derivation induced by an element $a \in R^\alpha G$.
For every $r \in R$ we have
\[
0 = d(r) = ar - ra,
\]
so it follows that $a \in Z ^\alpha G$. Thus, given $t \in T$ and $g \in G$, we have
\[
d_T(tg) =t d(g) =t ( a\overline{g} - \overline{g}a) = a (t\overline{g})  -(t\overline{g}) a  
\]
and hence $d_T$ is also inner.

Now suppose that $d_T$ is inner. Then, for every $g \in G$, we have $d(g) = d_T(g) = a \overline{g} - \overline{g} a$. Thus, since $d_T$ is a $T$-derivation, it follows that $a \in Z(T)^\alpha G$. Therefore,

\begin{equation}
    a = \sum_{g\in G} a_h \overline{g}
\end{equation}
with $a_h \in Z(T)$. Then $d(g)\in R^\alpha G \Longrightarrow a\overline{g} - \overline{g}a \in R^\alpha G $, so
\begin{equation}
    \left(\sum_{h\in G} a_h \overline{h} \right) \overline{g} - \overline{g} \left( \sum_{h\in G} a_h \overline{h}\right) = \sum_{h\in G} r_g \overline{h} 
\end{equation}
from where
\begin{equation}
    \sum_{h\in G} a_h \alpha(h,g) \overline{hg}  -   \sum_{h\in G} \alpha(g,h) a_h \overline{gh} = \sum_{h\in G} r_g \overline{h} 
\end{equation}
and as $hg = gt \Longrightarrow t = g^{-1} h g$ we have that
\begin{equation}
    a_{h}\alpha(h,g) - \alpha(g,g^{-1}hg)a_{g^{-1} h g} \in R
\end{equation}
For every $h \in G$, we then have either $a_h \in R$ or $a_h \notin R$, in which case we would obtain $a_{g^{-1} h g} \neq 0$ for every $g \in G$, implying that $h$ must have a finite number of conjugates. Let $\chi$ denote the subgroup of elements in $G$ that have a finite set of conjugates.

Then $\beta = \sum_{g\not \in \chi} a_g \overline{g}$, $\psi = \sum_{g \in \chi} a_g \overline{g}$. From where $a=\beta + \psi$, with $\beta \in R^\alpha G$, and we can focus on $\psi$.

If $h\in \chi$, then $a_{h}\alpha(h,g) - \alpha(g,g^{-1}hg)a_{g^{-1} h g} = r_{hg} \in R$, from where 

\begin{equation}
    a_{g^{-1} h g} =\alpha(g,g^{-1}hg) ^{-1} (a_{h}\alpha(h,g) - r_{hg})
\end{equation}

Let $C_{h,g} = \frac{\alpha(h,g)}{\alpha(g,g^{-1}hg) }$, $r'_{h,g} = \frac{r_{gh}}{\alpha(g,g^{-1}hg)}$ then

\begin{equation}
    a_{g^{-1} h g} = C_{h,g}a_{h} - r'_{h,g}
\end{equation}

If $\Psi$ are represented of the conjugate classes in $\chi$, then
\begin{align*}
    \sum_{h\in \chi} a_h \overline{h} 
    & = \sum_{h\in \Psi}  \sum_ {x\in G} a_{x^{-1}hx} \overline{x^{-1}hx} 
    \\ & = \sum_{h\in \Psi} ( a_h \overline{h} + a_{x^{-1}hx} \overline{x^{-1}hx} + \cdots ) 
    \\ & = \sum_{h\in \Psi} ( a_h \overline{h} + (C_{h,x} a_h - r'_{h,g}) \overline{x^{-1}hx} + \cdots )
    \\ & =\sum_{h\in \Psi}  a_h \overline{h} + C_{h,x} a_h \overline{x^{-1}hx} - r'_{h,g} \overline{x^{-1}hx} + \cdots 
    \\ & =\sum_{h\in \Psi}  a_h ( \overline{h} + C_{h,x}  \overline{x^{-1}hx} \cdots )  - (\sum_{x\in G}  r'_{h,x} \overline{x^{-1}hx} + \cdots )
    \\ & =\sum_{h\in \Psi} a_h K_h + \delta 
\end{align*}

Where $\delta \in R^ \alpha G$, and

\begin{equation}
    K_h = \sum_{x\in G} C_{h,x}  \overline{x^{-1}hx}
\end{equation}

with $C_{h,h}=1$. Finally, note that $x^{-1}hx g = gt$ so $t= g^{-1}x^{-1}hx g$, which is a conjugate of $h$, so it is enough to study $\alpha_h K_h$. For this

\begin{align*}
    \alpha_h K_h \overline{g} - \overline{g}\alpha_h K_h = & \alpha_h \sum_{x\in G} C_{h,x} \alpha(x^{-1}hx,g)\overline{x^{-1}hxg} - C_{h,x} \alpha(g,x^{-1}hx) \overline{g x^{-1}hx} \\
    = & \alpha_h \sum_{x\in G} ( C_{h,x} \alpha(x^{-1}hx,g)- C_{h,xg} \alpha(g,g^{-1}x^{-1} h xg) ) \overline{x^{-1}hx}
\end{align*}

Where we can find

\begin{equation}
    \alpha_h \left( C_{h,x} \alpha(x^{-1}hx,g)- C_{h,xg} \alpha(g,g^{-1}x^{-1} h xg \right) \in R
\end{equation}

Here, if 
\[
T(h,x,g) = C_{h,x} \, \alpha(x^{-1} h x, g) - C_{h,xg} \, \alpha(g, g^{-1} x^{-1} h x g) \neq 0,
\] 
then this element belongs to the field $k$, so it is invertible and $\alpha_h \in R$. Otherwise, if it is zero for all $x, g \in G$, we have $K_h \in Z(R^\alpha G)$, and therefore removing the term $\overline{h}$ from $a$ does not change the value of the conjugation. In this way, conjugating with 
\[
a' = \beta + \delta + \sum_{\substack{h \\ T(h,x,g) \neq 0,\, x,g \in G}} \alpha_h \, \overline{h} \in R
\] 
is equivalent to conjugating with $a$, and we obtain the result.

\end{proof}

In \cite[Section 1.4]{Passman1989} the following theorem is proven for twisted group algebras. 
\begin{theorem}[Maschke's Theorem for twisted group algebras]
Let $F$ be a field and $G$ a finite group. 
If $\operatorname{char}(F) \nmid |G|$, then the twisted group algebra 
$F^{\alpha}G$ is semisimple for every $2$-cocycle $\alpha : G \times G \to F^{\times}$. 
\end{theorem}

As a result, we can obtain the following definition and lemma. 

\begin{definition}
    Let $G$ be a group, $H$ an abelian group and $\alpha:G\times G \longrightarrow H$ a $2-$cocycle. The $\alpha-$center, noted by $Z_\alpha(G)$, is given by
    \begin{equation}
        Z_\alpha(G) = \{ a\in Z(G) ; \alpha(a,x) = \alpha(x,a), \forall x \in G\}
    \end{equation}
\end{definition}

\begin{lemma}
    Let $R$ a $k-$algebra, $G$ a group and $\alpha:G\times G \longrightarrow k^*$ . Then $Z_\alpha(G) \subset Z(R^\alpha G)$ and for all $R-$derivation $d$ in $R^\alpha G$, $z \in Z_ \alpha(G) $ torsion element of order coprime with the characteristics of $R$, we have $d(z)=0$.
\end{lemma}
\begin{proof}
    For the first part, we have that if $z\in Z_\alpha(G)$, then for all $x\in G$, we have that 
    \begin{equation}
        \overline{z} \cdot \overline{x} = \alpha(z,x) \overline{zx} = \alpha(x,z) \overline{xz} = \overline{x} \cdot \overline{z}
    \end{equation}
    As we wanted to see. Moreover, let $z \in Z_ \alpha(G) $ torsion, then there exists $m$ such that $z^m$. We have that
    \begin{align*}
        0 = &  0 \frac{1}{\prod_{i=1}^{m-1}\alpha(z^i,z)}d(\overline{z^m})  \\
        =  &        d(\overline{z}^m) = \sum_{i=1}^{m}\prod_{j=1}^{i-1} \overline{z} d(z) \prod_{j=i+1}^{m} \overline{z}  \\
        =  &  \sum_{i=1}^{m} \overline{z}^{m-1} d(z) \\
        =  &  m \overline{z}^{m-1} d(z)
    \end{align*}

and by hipotesys $d(z)=0$.
\end{proof}

\begin{theorem}
Let $R$ be a semiprime $k-$algebra and $G$ a torsion group such that 
$[G : Z_\alpha(G)] < \infty$, and $\alpha: G\times G \longrightarrow k^*$ a $2-$cocycle. Suppose that either $\operatorname{char} R = 0$ or, for every characteristic $p$ of $R$, 
$p \nmid o(g)$ for all $g \in G$. Then every $R$-derivation of $R^\alpha G$ is inner.
\end{theorem}

\begin{proof}
First, assume that $G$ is a finite group and that $|G|$ is invertible in $R$,  for every characteristic $p$ of $R$. Because of the proposition above, it is enough to prove the result when $R$ is a commutative semiprime $k-$algebra. We claim that we 
also may assume that $R$ is Noetherian. In fact, let $R'$ be the subalgebra generated by the finitely many elements of $R$ which occur as coefficients of elements in $d(G)$. Then $d$ restricts to a derivation of $R'^\alpha G$ and $R'$ is semiprime and Noetherian and it follows again from our proposition that it suffices to prove that $d$ is inner in $R'^\alpha G$.

Let $P_1, P_2, \dots, P_n$ be the finitely many minimal prime ideals of $R$ and let 
$F_i$ be an algebraically closed field containing $R/P_i$, $1 \leq i \leq n$. 
Since $R$ is semiprime, we have that $\bigcap_{i=1}^n P_i = 0$; hence, $R$ can be 
embedded in $T = \bigoplus_{i=1}^n F_i$. Thus, by the proposition, it is enough to prove 
that $d_T$ is inner in $T^\alpha G$.

Note that if the characteristic of $F_i$ is a prime integer $p_i$, then $p_i \in F_i$ so 
$p_i \cdot 1_{R/P_i} = 0$ and thus $p_i$ is a characteristic of $R$; therefore, 
$p_i \nmid |G|$, $1 \leq i \leq n$. Hence, Maschke's Theorem shows that 
$T^\alpha G$ is a direct sum of full matrix $k$-algebra over fields, say 
\[
T^\alpha G = I_1 \oplus \cdots \oplus I_k,
\]
where each $I_j$ is generated, as an ideal, by a central idempotent. Since it is easily 
seen that $d_T(e) = 0$ for every central idempotent $e \in T^\alpha G$, it follows that 
$d_T(I_j) \subseteq I_j$ and thus $d_T$ gives, by restriction, a derivation of each 
component, which is inner, induced by an element $a_j \in I_j$ 
(see \cite[p.~100]{Herstein1968}). Consequently $d_T$ is the inner derivation of $T^\alpha G$ induced by 
$a = a_1 + \cdots + a_k$.

Now, we have that  $Z_\alpha(G)\subset Z(R^\alpha G)$, as
\begin{equation}
        \overline{z} \cdot \overline{x} = \alpha(z,x) \overline{zx} = \alpha(x,z) \overline{xz} = \overline{x} \cdot \overline{z}
    \end{equation}
    
And also it is true that $d(Z_\alpha (G))=0$ for the previous lemma.  

Let $X = \{ g_1, \dots, g_n \}$ be a transversal of $Z_{\alpha}(G)$ in $G$. 
For every index $i$, $1 \leq i \leq n$, we write:
\[
d(\overline{g_i}) = \sum_{j,k} a_{ijk} \overline{z_{ijk} g_k}, \quad z_{ijk} \in Z_{\alpha}(G), \; a_{ijk} \in R.
\]

Also, for $i,j = 1, \dots, n$ let $g_i g_j = c_{ij} g_k$, $c_{ij} \in Z_{\alpha}(G)$. 
Denote by $H$ the subgroup of $G$ generated by all the elements 
$z_{ijk}, c_{ij}, g_i$. Since $Z_{\alpha}(G)$ is abelian and $G$ is torsion, it follows 
that $H$ is finite. Also, the restriction $d|_{R^\alpha H}$ is an $R$-derivation of $R^\alpha H =R^{\alpha|_H} H $. 
By the first part, there exists an element $a \in R^\alpha H$ such that $d|_{R^\alpha H}$ is the inner 
derivation induced by $a$.

Now, given an element $g \in G$, write $g = z g_i$, with $z \in Z_{\alpha}(G)$, $1 \leq i \leq n$. 
Then:

\begin{align*}
    d(\overline{g}) = & d(\overline{z g_i}) \\ =  & \frac{1}{\alpha(z,g_i)} d(\overline{z}\cdot \overline{g_i})  \\ =  & \frac{1}{\alpha(z,g_i)} \overline{z} d(\overline{g_i}) \\ =  &  \frac{1}{\alpha(z,g_i)}\overline{z}(a\overline{g_i} - \overline{g_i} a) \\ =  & a \left( \frac{1}{\alpha(z,g_i)}\overline{z} \cdot \overline{ g_i}  \right) - \left( \frac{1}{\alpha(z,g_i)}\overline{z}  \overline{g_i} \right)a \\ =  & a\overline{g} - \overline{g}a.
\end{align*}

Consequently, $d$ is inner in $RG$, induced by $a$.
\end{proof}

\section{Application to abelian group}
In this section we study derivations on commutative twisted groups rings. Our approach is based on analyzing the different twists determined by $2$-cocycles and examining how these twists influence the structure of the derivations.

Given $R^\alpha G$ a twisted group ring  and $H\leq G$, the subring generated by the subgroup $H$ is $R^{\alpha} H =R^{\alpha|_H} H$ where $\alpha|_H$ is the restriction $\alpha$ to $H$. 

\begin{theorem}
Let $R$ be a commutative unital ring. $G$ a group and $\alpha: G\times G \longrightarrow R^*$ . Let $H\subset Z_\alpha(G)$ be a torsion central subgroup of a group $G$, where the order of $h$ is not a $0$-divisor in $R$ for all $h \in H$. Then $d(R) = \{0\}$ if and only if $d(R^\alpha H) = \{0\}$ for all $d \in \mathrm{Der}(R^\alpha G)$.
\end{theorem}

\begin{proof}
Let $d$ be any element of $\mathrm{Der}(R^\alpha G)$. Assume that $d(R) = \{0\}$. Let $h \in H$ be an element of order $s$. Then, in $R^\alpha H$ we have 
\begin{equation}
    \overline{h}^s = \alpha(h,h) \alpha(h,h^2)\cdots \alpha(h,h^{s-1}) \overline{1} = \prod_{j=1}^{s-1} \alpha(h,h^j)\overline{1}
\end{equation}

Applying $d$ to the previous equation implies  
\begin{equation}
    0 = d(\overline{1}) = \prod_{j=1}^{s-1} \alpha(h,h^j) d(1) =  d\left(\prod_{j=1}^{s-1} \alpha(h,h^j)\overline{1} \right) =  d(\overline{h}^s) = s \overline{h} ^{s-1} d(\overline{h})  
\end{equation}
Where in the last identity we have used $H$ is central. Therefore, $s$ is not a $0$ divisor and $\overline{h} ^{s-1}$  is invertible, leading to $d(\overline{h}) = 0$. So $d(h) = 0$ for any $d \in \mathrm{Der}(R^\alpha G)$.

Let $\alpha = \sum_{h \in H} a_h h$ be any element of $R^\alpha H$. Then,
\[
d(\alpha) = d\left( \sum_{h \in H} a_h h \right) = \sum_{h \in H} d(a_h h) = \sum_{h \in H} a_h d(h) = \sum_{h \in H} a_h \cdot 0 = 0,
\]
by Leibniz's rule, since $d(R) = \{0\}$. Hence, $d(R^\alpha H) = \{0\}$.

The converse is immediate.
\end{proof}

\begin{corollary}
\begin{itemize}
    \item[(i)] Let $G$ be a finite abelian group and $F$ either the field of rational numbers or an algebraic extension of the rationals. Then $F^\alpha G$ has no nonzero derivations.
    \item[(ii)] Let $H$ be a $p$-regular subgroup of a finite abelian group $G$ and $F = \mathbb{F}_{p^n}$. Then all derivations of $F^\alpha G$ are $F^{\alpha} H$-derivations.
\end{itemize}
\end{corollary}

\begin{proof}
For part $i)$, let $H = G$, then the order of all elements in $G$ is a finite number and therefore invertible, so it is not a $0$ divisor.
For case $ii)$ the order of $h\in H$ is not a multiple of $p$, and as a result it is not a $0$ divisor. This lead to $d\in Der(F^\alpha G)$ and $d(F^{\alpha} H) = 0$ as we wanted to prove. 
\end{proof}

In \cite[Theorem 2.1]{huang2020explicit}, a complete set of representatives of $k$-cocycles is computed. As a consequence, we have the following corollary.

\begin{corollary}
    Let $R$ be a $k$-algebra, Let $G$ be a finite abelian group and let $\alpha: G \times G \longrightarrow k^*$. Then the twisted group ring $R^\alpha G$ is commutative.
\end{corollary}

\begin{theorem} \label{EstructuraDerivacion}
    Let $K$ be a finite field of positive characteristic $p$. Let $G \cong H \times X$ be a finite abelian group, where $H$ is a $p$-regular group and $X$ is a $p$-group with the following presentation:
\[
X = \langle x_1, \ldots, x_n \mid x_k^{p^{m_k}} = 1,\; [x_k, x_l] = 1 \text{ for all } k, l \in \{1, 2, \ldots, n\} \rangle,
\]
where $n, m_k \in \mathbb{N}$. Let $\alpha:G \times G \longrightarrow K$ a $2-$cocycle such that $H\subset Z_\alpha(G)$. For $i, j \in \{1, \ldots, n\}$, define the map
\[
f_i \colon \{x_1, \ldots, x_n\} \to K^\alpha G \quad \text{by} \quad
f_i(x_j) = 
\begin{cases}
1, & \text{if } i = j, \\
0, & \text{otherwise}.
\end{cases}
\]
Then $f_i$ can be uniquely extended to a derivation of $KG$, denoted by $\partial_i$. Moreover, $\mathrm{Der}(K^\alpha G)$ is a vector space over $K$ with basis
\[
\{ \overline{g} \partial_i \mid g \in G,\; i = 1, \ldots, n \}.
\]
\end{theorem}
\begin{proof}
    Let $S=\{x_1,\cdots, x_n\}$ and $f: S \longrightarrow K^\alpha G$. To prove that $f$ can be extended to a derivation, we have to see that it respect the relations over $x_i$. 
    \begin{enumerate}
        \item Let $x_i,x_j\in S$, and we study $[a,b]=1$
        \begin{align*}
            f^*(\overline{[x_i,x_j]}) & =  f^*(\overline{x_i^{-1}x_j^{-1}x_ix_j} ) \\ 
            & = C_{\overline{[x_i,x_j]}} f^*(\overline{x_i^{-1}} \cdot \overline{x_j^{-1}} \cdot \overline{x_i}\cdot  \overline{x_j} ) \\ 
            & \sim f^*(\overline{x_i^{-1}} \cdot \overline{x_j^{-1}} \cdot \overline{x_i}\cdot  \overline{x_j} )  \\
            & \sim f^*( \overline{x_i^{-1}} )  \cdot \overline{x_j^{-1}} \cdot \overline{x_i}\cdot  \overline{x_j} \\ 
            & +\overline{x_i^{-1}} \cdot f^*(\overline{x_j^{-1}}) \cdot \overline{x_i}\cdot  \overline{x_j} \\ 
            & + \overline{x_i^{-1}} \cdot \overline{x_j^{-1}} \cdot f^*(\overline{x_i}) \cdot  \overline{x_j} \\ 
            & + \overline{x_i^{-1}} \cdot \overline{x_j^{-1}} \cdot \overline{x_i}\cdot  f^*(\overline{x_j})\\
            & = -\frac{1}{\alpha(x_i^{-1}, x_i)} \overline{x_i^{-1}}f^*(\overline{x_i} )\overline{x_i^{-1}}\cdot \overline{x_j^{-1}} \cdot \overline{x_i}\cdot  \overline{x_j}  \\
            & -\frac{1}{\alpha(x_j^{-1},x_j)}\overline{x_i^{-1}} \cdot \overline{x_j^{-1}}f^*(\overline{x_j}) \overline{x_j^{-1}}\cdot \overline{x_i}\cdot  \overline{x_j} \\ 
            & + \overline{x_i^{-1}} \cdot \overline{x_j^{-1}} \cdot f^*(\overline{x_i}) \cdot  \overline{x_j} \\ 
            & + \overline{x_i^{-1}} \cdot \overline{x_j^{-1}} \cdot \overline{x_i}\cdot  f^*(\overline{x_j})\\
            & = -\alpha(x_i,x_i^{-1}) f^*(x_j) \overline{x_j^{-1}} \\
            & - \alpha(x_j,x_j^{-1}) f^*(x_i) \overline{x_i^{-1}} \\
            & + \alpha(x_i,x_i^{-1}) f^*(x_j) \overline{x_j^{-1}} \\
            & + \alpha(x_j,x_j^{-1}) f^*(x_i) \overline{x_i^{-1}} \\
            & = 0
        \end{align*}
    \item Now, we need to see the relation $a^{p^m_k} =1$
    \begin{align*}
        f^*(x^{p^{m_k}}) & \sim \sum_{i=1}^{p^{m_k}} \prod_{j<i} \overline{x} f^*(\overline{x}) \prod_{j>i}\overline{x}\\
        & = p^{m_k} f^*(x) \overline{x}^{p^m_k-1} \\
        & = 0
    \end{align*}
    \end{enumerate}
    so $f_i$ can be extended to a derivation. For the second part, they are linearly independent; let $g\in G, i \in \{0,\dots, n-1\}$
    \begin{equation}
        \overline{g} \partial_i = \sum_{\overline{h}\partial_j \not = \overline{g}\partial_i} a_i \overline{h}\partial_j
    \end{equation}
    As the elements of $G$ in $K^\alpha G$ are linearly independet, the last equation reduce to
    \begin{equation}
        \overline{g} \partial_i = \sum_{j \not = i} a_j \overline{g}\partial_j
    \end{equation}
    And evaluate on $x_i$, we obtain $1= 0$, a contradiction. To prove they are a generator set, let $d$ be a derivation, then
    \begin{equation}
        d(x) = \sum_{g\in G} a_{g,x} \overline{g}
    \end{equation}
    with $a_{g,x} \in K$ for all $x\in G$. Then
    \begin{align*}
        d(x) = & \sum_{g\in G} a_{g,x} \overline{g} \\ = & \sum_{g\in G} a_{g,x} \overline{g}\partial_{x} (x) \\ = & \sum_{h\in G} \sum_{g\in G} a_{g,h} \overline{g}\partial_{h} (x) \\ = & \left(\sum_{h\in G} \sum_{g\in G} a_{g,h} \overline{g}\partial_{h} \right) (x)
    \end{align*}
    
    So we have the following identity 
    
    \begin{equation}
        d = \sum_{h\in G} \sum_{g\in G} a_{g,h} \overline{g}\partial_{h}
    \end{equation}
    with $a_{g,h}\in K$
\end{proof}

\begin{corollary}
    The derivations of finite commutative twisted group ring $\mathbb{F}_{p^n}^\alpha G$ are either the zero derivation, in the semisimple case, or can be decomposed, as in Theorem \ref{EstructuraDerivacion}, as the sum of derivations of the group algebras of the direct cyclic factors of $G$.
\end{corollary}

\section{Application to dihedral group}

Let $n$ be an integer greater than $2$ and let $D_{2n}$ denote the dihedral group with $2n$ elements and presentation $〈r, s | r^n = s^2 = (rs)^2 = 1〉$. In \cite{CREEDON2019247} the authors study the derivation of the case $\mathbb{F}_{2^m} D_{2n}$. Here we will generalize their result to the case of twisted dihedral group.  

First of all, we recall the second cohomology group of the dihedral group over a finite field. Using the results in \cite{zbMATH03935317}, it is possible to determine that $H^2(D_{2n},\mathbb{F}_{2^m}^* ) = 0 $, and if $q=p^k$, $p\not = 2$ prime,  then if $n$ is odd, we have $H^2(D_{2n},\mathbb{F}_{q}^* ) = <\alpha_3>  \equiv C_2$, and when $n$ is even $H^2(D_{2n},\mathbb{F}_{q}^* ) = <\alpha_1,\alpha_2, \alpha_3>  \equiv C_2 \times C_2 \times C_2$, where
\begin{align*}
    \alpha_1: D_{2m} \times D_{2m} & \longrightarrow \mathbb{F}_{q}^* \\
    (r^as^b, r^cs^d)& \longmapsto(-1)^{bc}
\end{align*}

\begin{align*}
    \alpha_2: D_{2m} \times D_{2m} & \longrightarrow \mathbb{F}_{q}^* \\
    (r^as^b, r^cs^d)& \longmapsto (-1)^{ad}
\end{align*}

\begin{align*}
    \alpha_3: D_{2m} \times D_{2m} & \longrightarrow \mathbb{F}_{q}^* \\
    (r^as^b, r^cs^d)& \longmapsto (-1)^{bd}
\end{align*}
To simplify the notation, let $\alpha=g^{\frac{q-1}{2}}$.

In this work, we will deal with $\mathbb{F}_q^{\alpha}D_{2n}$ with $n$ odd, and the general relations in case $n$ even. 

To begin with, let recall the first Hochschild cohomology group for the case of ring over itself.
\begin{definition}
    Let $A$ be an associative ring (or algebra over a commutative ring $k$), and let $M$ be an $A$-bimodule. The \emph{Hochschild cohomology groups} $HH^n(A, M)$ are defined as the cohomology of the cochain complex $(C^\bullet(A, M), \delta)$ where

\[
C^n(A, M) := \operatorname{Hom}_k(A^{\otimes n}, M), \quad n \ge 0,
\]

and the coboundary map $\delta: C^n(A,M) \to C^{n+1}(A,M)$ is
\begin{align*}
    (\delta f)(a_1, \dots, a_{n+1}) 
= & a_1 f(a_2, \dots, a_{n+1}) 
+ \sum_{i=1}^{n} (-1)^i f(a_1, \dots, a_i a_{i+1}, \dots, a_{n+1}) +
 \\ & + (-1)^{n+1} f(a_1, \dots, a_n) a_{n+1}.
\end{align*}

Then the Hochschild cohomology is
\[
HH^n(A, M) := H^n(C^\bullet(A, M), \delta) 
= \frac{\ker(\delta: C^n \to C^{n+1})}{\operatorname{im}(\delta: C^{n-1} \to C^n)}.
\]
\end{definition}

We will focus on the first Hochchild cohomology group, which under the previous consideration is defined as
\[
HH^1(A, M) = \frac{
\{ f: A \to M \text{ linear } \mid f(ab) = af(b) + f(a)b \}
}{
\{ f:A \to M ; f(a) = ma - am \text{ for some } m \in M \}
}.
\]

In particular, in case $M=A$, we have that
\begin{equation}
    HH^1(A,A) = \frac{Der(A)}{Inn(A)}
\end{equation}

In fact, $\mathbb{F}_{p^m}G$ is a semisimple finite-dimensional algebra over a perfect field if $p \nmid |G|$ and therefore separable. It is a well-known result that the first Hochschild cohomology group of separable ring vanishes. \cite{Hochschild1945}

Therefore, we have the following corollary.
\begin{corollary}
    Let $\mathbb{F}_{p^m}D_{2n}$ with $p\nmid n$. Then all $\mathbb{F}_{p^m}$ derivations are inner.
\end{corollary}
This result will be proved important, as the condition $p\nmid n$ will play a predominant role. In addition, we are able to prove that
\begin{theorem}
    Let $D_{2n}$ with $n$ odd, let $\mathbb{F}_q$ a finite field of order $q=p^k$ with $k | n$, and let $\alpha \in H^2(G,\mathbb{F}_q)$ a $2-$cocycle. Then $HH^1(\mathbb{F}_q^\alpha D_{2n}) =  \mathbb{F}_q^{\frac{n-1}{2}}$
\end{theorem}

\subsection{$D_n$ with $n$ even}
Let $D_{2n}$ with $n$ even, $q=p^m$ with $p$ odd prime, and let $\mathbb{F}_q^{\alpha} D_{2n}$. Let $f$ a map defined over $r,s$ in $\mathbb{F}_q^{\alpha} D_{2n}$
\begin{align*}
    f(r) & = \sum_{i=0}^{n-1} \gamma _i \overline{r^i} + \delta_i \overline{r^is} \\
    f(s) & = \sum_{i=0}^{n-1} h _i \overline{r^i} + t_i \overline{r^is}
\end{align*}
also, to simplify the notation, let 
\begin{equation}
    f(rs) = \alpha_2(r,s) \left( f(r) \overline{s} + \overline{r} f(s)\right) = \sum_{i=0}^{n-1} w _i \overline{r^i} + A_i \overline{r^is} 
\end{equation}

We will see which conditions must be satisfied so that $f$ can be extended to a derivation. Let's start with $f^*(r^n) =0 $

\begin{align*}
    f^*(r^n) & =  \sum_{k=0}^{n-1} \left( \prod_{i=0}^{k-1} \overline{r}\right) f(r) \left( \prod_{i=k+1}^{n-1} \overline{r}\right) \\
    & = \sum_{k=0}^{n-1} \overline{r}^{k} \left( \sum_{i=0}^{n-1} \gamma _i \overline{r^i} + \delta_i \overline{r^is} \right) \overline{r}^{n-k-1} \\
    & =  \sum_{k=0}^{n-1} \sum_{i=0}^{n-1} \gamma _i \overline{r^i} + \delta_i  \alpha(r^k,r^is)\alpha(r^{k+i}s,r^{n-k-1}) \overline{r^{i+k-(n-k-1)}s} 
\end{align*}
That give rise to
\begin{align*}
    n \gamma _i & = 0 \\
    \sum_{k=0}^{n-1} \sum_{i=0}^{n-1} \delta_i  \alpha(r^k,r^is)\alpha(r^{k+i}s,r^{n-k-1}) \overline{r^{i+2k+1}s} & = 0
\end{align*}
From where $p| n $ and $\gamma$ is free, or $ p \nmid n$ and $\gamma_i=0$. For the second equality, lets associate the terms corresponding to the same basis elements $\overline{r^\Omega}$.

If $\Omega = i+2k+1 \mod n$, then $2k = \Omega -i-1 \mod n$. As $n$ is even, $n=2n'$, then we have: If $\Omega \equiv_2 i$ there is no solution, and in other case $k_{1}(i,\Omega) = \frac{ \Omega -i-1}{2}, k_{2}(i,\Omega) = \frac{ \Omega -i-1}{2} + n'$. Therefore the last equality is
\begin{align*}
    \sum_{ i = 0; \Omega \equiv_2 i +1} ^{n-1} & \delta_i ( \alpha(r^{k_1(i,\Omega)},r^is)\alpha(r^{k_1(i,\Omega)+i}s,r^{n-k_1(i,\Omega)-1})+ \\ & + \alpha(r^{k_2(i,\Omega)},r^is)\alpha(r^{k_2(i,\Omega)+i}s,r^{n-k_2(i,\Omega)-1})) = 0
\end{align*}

From where $\delta = (\delta_i)$ is a solution to $AX=0$ with $A=(a_{\Omega,i}) _{\Omega,i} $ where 
\begin{align}
    a_{\Omega,i}= & \alpha(r^{k_1(i,\Omega)},r^is)\alpha(r^{k_1(i,\Omega)+i}s,r^{n-k_1(i,\Omega)-1})+\\  & +  \alpha(r^{k_2(i,\Omega)},r^is)\alpha(r^{k_2(i,\Omega)+i}s,r^{n-k_2(i,\Omega)-1})
\end{align}

if $\Omega \equiv_2 i+1$ and $0$ other. 

Now lets see $f^*(s^2) =0$.
\begin{align*}
    -f(s) & = \frac{1}{\alpha(s,s)} \overline{s} f(s) \overline{s} \\
    & = \frac{1}{\alpha(s,s)} \overline{s} \left(\sum h_i \overline{r^i} + t_i \overline{r^i s} \right) \overline{s} \\
    & = \frac{1}{\alpha(s,s)} \left(\sum h_i \alpha(r^i,s) \alpha(s,r^is)\overline{r^{-i}} + t_i \alpha(r^is,s) \alpha(s,r^i)\overline{r^{-i} s} \right)  \\
    & 
\end{align*}

From where
\begin{align*}
    h_{-i} \alpha(s,s) + \alpha(r^i,s) \alpha(s,r^is)h_{i} = 0 \\
    t_{-i} \alpha(s,s)  +\alpha(r^is,s) \alpha(s,r^i) t_{i} = 0 \\
\end{align*}

Now we have the relation $f^*(rsrs) = 0$, which is equivalent to $-f^*(rs) = \frac{1}{\alpha(rs,rs)} \overline{rs} f^*(rs) \overline{rs}$

\begin{align*}
    - f^*(rs) & = \frac{1}{\alpha(rs,rs)} \overline{rs} f^*(rs) \overline{rs} \\
    & = \frac{1}{\alpha(rs,rs)} \overline{rs} \left(\sum_{i=0}^{n-1} w_i \overline{r^i} + A_i \overline{r^is} \right) \overline{rs} \\
    & = \frac{1}{\alpha(rs,rs)}\left(\sum_{i=0}^{n-1} w_i \alpha(rs,r^i) \alpha(r^{1-i},rs) \overline{r^{-i}} + A_i \alpha(rs,r^is)\alpha(r^{1-i},rs)\overline{r^{2-i}s} \right)\\
\end{align*}
so
\begin{align*}
    w_i\alpha(rs,r^i) \alpha(r^{1-i},rs)  + w_{-i}\alpha(rs,rs) & = 0\\ 
    A_i \alpha(rs,r^is)\alpha(r^{1-i},rs) + A_{2-i}\alpha(rs,rs) & = 0
\end{align*}
and in addition
\begin{align*}
    f(rs) = & \frac{1}{\alpha(r,s)} \left(f(r) \overline{s} + \overline{r} f(s) \right) \\ 
    = & \frac{1}{\alpha(r,s)} \left( \left( \sum_{i=0}^{n-1} \gamma _i \overline{r^i} + \delta_i \overline{r^is}\right)\overline{s} + \overline{r} \left( \sum_{i=0}^{n-1} h _i \overline{r^i} + t_i \overline{r^is} \right) \right) \\
    = & \frac{1}{\alpha(r,s)}  \left(  \sum_{i=0}^{n-1} \gamma _i \alpha(r^i,s) \overline{r^is} + \delta_i \alpha(r^is,s) \overline{r^i}+ \sum_{i=0}^{n-1} h _i \alpha(r,r^i) \overline{r^{i+1}} + t_i \alpha(r,r^is)\overline{r^{i+1}s} \right)
\end{align*}
from where 
\begin{align*}
    w_i \alpha(r,s) &=   \delta_i \alpha(r^is,s) + h_{i-1} \alpha(r,r^{i-1}) \\
    A_i \alpha(r,s)&= \gamma _i \alpha(r^i,s) + t_{i-1} \alpha(r,r^{i-1}s)
\end{align*}
achieving the equalities;
\begin{align*}
    (\delta_i \alpha(r^is,s) + h_{i-1} \alpha(r,r^{i-1}))\alpha(rs,r^i) \alpha(r^{1-i},rs)  +(\delta_{-i} \alpha(r^{-i}s,s) + & \\ + h_{-i-1} \alpha(r,r^{-i-1}))\alpha(rs,rs) & = 0\\ 
    (\gamma _i \alpha(r^i,s) + t_{i-1} \alpha(r,r^{i-1}s)) \alpha(rs,r^is)\alpha(r^{1-i},rs) + & \\ + \left(\gamma _{2-i} \alpha(r^{2-i},s) + t_{1-i} \alpha(r,r^{1-i}s) \right)\alpha(rs,rs) & = 0
\end{align*}

This leads to
\begin{theorem}
    Let $D_{2n}$ be the dihedral group with $n$ even, let $\mathbb{F}_{p^m}$ with $p$ prime, and let $\alpha\in H^2(D_{2n},\mathbb{F}_{p^m})$ be a $2-$cocycle. Let $f:\mathbb{F}_q^{\alpha_1}D_{2n} \longrightarrow \mathbb{F}_q^{\alpha_1}D_{2n}$ defined over $\overline{r},\overline{s}$ given by: 
\begin{align*}
    f(r) & = \sum_{i=0}^{n-1} \gamma _i \overline{r^i} + \delta_i \overline{r^is} \\
    f(s) & = \sum_{i=0}^{n-1} h _i \overline{r^i} + t_i \overline{r^is}
\end{align*}
can be extended as a derivation if and only if the following restrictions are satisfied
\begin{align*}
    A\delta & =0 \\
n \gamma _i & = 0 \\
    h_{-i} \alpha(s,s) + \alpha(r^i,s) \alpha(s,r^is)h_{i} & = 0 \\
    t_{-i} \alpha(s,s)  +\alpha(r^is,s) \alpha(s,r^i) t_{i} &= 0 \\
    \delta_i \alpha(r^is,s)\alpha(rs,r^i) \alpha(r^{1-i},rs) + h_{i-1} \alpha(r,r^{i-1})\alpha(rs,r^i) \alpha(r^{1-i},rs) +  & 
    \\  +\delta_{-i} \alpha(rs,rs)\alpha(r^{-i}s,s) + h_{-i-1} \alpha(r,r^{-i-1})\alpha(rs,rs) & = 0 \\
        (\gamma _i \alpha(r^i,s) + t_{i-1} \alpha(r,r^{i-1}s)) \alpha(rs,r^is)\alpha(r^{1-i},rs) + & \\ + \left(\gamma _{2-i} \alpha(r^{2-i},s) + t_{1-i} \alpha(r,r^{1-i}s) \right)\alpha(rs,rs) & = 0
\end{align*}
Such that $A=(a_{\Omega,i}) _{\Omega,i} $ with

\[
a_{\Omega,i} =
\begin{cases}
\begin{aligned}
& \alpha(r^{k_1(i,\Omega)}, r^i s)\,
   \alpha(r^{k_1(i,\Omega)+i} s, r^{n-k_1(i,\Omega)-1}) \\
&\quad + \alpha(r^{k_2(i,\Omega)}, r^i s)\,
   \alpha(r^{k_2(i,\Omega)+i} s, r^{n-k_2(i,\Omega)-1}),
\end{aligned}
& \text{if } \Omega \equiv_2 i+1, \\[6pt]
0, & \text{other}.
\end{cases}
\]
\end{theorem}

Note that the restriction $t_{-i} \alpha(s,s)  +\alpha(r^is,s) \alpha(s,r^i) t_{i} =0$ can be expressed in terms of $\gamma_i$ using the last identity, but we have preferred to express them like this as the substitutions makes the restrictions harder to follow. 
\begin{example} \label{DerivParn'impar}
    Let $D_{2n}$ with $n=2n'$, $n'$ odd, and $\mathbb{F}_q$ a finite field of $q$ elements. Let $\alpha_1$ be the $2-$cocycle given by
\begin{align*}
    \alpha_1: D_{2m} \times D_{2m} & \longrightarrow \mathbb{F}_{q}^* \\
    (r^as^b, r^cs^d)& \longmapsto(-1)^{bc}
\end{align*}
Then, note that 
\begin{align*}
   a_{i,\Omega}  =  & \alpha(r^{k_1(i,\Omega)}, r^i s)\,
   \alpha(r^{k_1(i,\Omega)+i} s, r^{n-k_1(i,\Omega)-1}) 
   + \\ & + \alpha(r^{k_2(i,\Omega)}, r^i s)\,
   \alpha(r^{k_2(i,\Omega)+i} s, r^{n-k_2(i,\Omega)-1})  \\   =& (-1)^{-k_1(i,\Omega)-1} + (-1)^{-k_2(i,\Omega)-1} \\ = & 0
\end{align*}
As $k_2(i,\Omega) -k_1(i,\Omega) = n'$ that is odd, so they have different parity. Therefore the relations in this case are
\begin{align*}
n \gamma _i & = 0 \\
    h_{-i}  +(-1)^ih_{i} & = 0 \\
    t_{-i}  + (-1)^{i+1}  t_{i} &= 0 \\
    \delta_i (-1)^{i}  +  (-1)^{i} h_{i-1}  +\delta_{-i}  + h_{-i-1}  &= 0 \\
        (-1)^{i}\gamma _i  - t_{i-1}   +\gamma _{2-i}  + (-1)^{i-1} t_{1-i}  & = 0
\end{align*}
and we can use that $(-1)^{i} t_{1-i} + t_{i-1} = 0$ to reduce
\begin{align*}
n \gamma _i & = 0 \\
    h_{-i}  +(-1)^ih_{i} & = 0 \\
    t_{-i}  + (-1)^{i+1}  t_{i} &= 0 \\
    (-1)^{i} \delta_i   +  (-1)^{i} h_{i-1}  +\delta_{-i}  + h_{-i-1}  &= 0 \\
        (-1)^{i}\gamma _i  +\gamma _{2-i}  & = 0
\end{align*}
Let suppose that $p|n$ and we will study the dimension of $Der(\mathbb{F}_q^{\alpha_1}D_{2n})$. For this, note that $t_0$ is free, and the relation
\begin{equation}
    t_{-i}  + (-1)^{i+1}  t_{i} = 0
\end{equation}
implies that $\frac{n-2}{2}$ are free and $t_{n'}=0$. Similarly with $h_{-i}  +(-1)^ih_{i}  = 0$ we have that $h_0=0$, $h_{n'}$ is free and $\frac{n-2}{2}$ of the other $h_i$ are free. Similarly, $\gamma_1$ is free, $\frac{n-2}{2}$ of the $\gamma_i$ are free and $\delta_0 = h_{n-1}$, $\delta_{n'}$ is free $\frac{n-2}{2}$ of the other are free. The total dimension as $\mathbb{F}_q$ vector space is $4\left(\frac{n-2}{2}+1 \right) = 2n$. 
\end{example}

\subsection{$\mathbb{F}_q^{\alpha_1}D_{2n}$ with $n$ odd} 

Let $D_{2n}$ with $n$ odd, $q=p^m$ with $p$ an odd prime, and let $\mathbb{F}_q$ be the finite field of q elements. Recall that the only non trivial $2-$cocycles is given by

\begin{align*}
    \alpha_3: D_{2m} \times D_{2m} & \longrightarrow \mathbb{F}_{q}^* \\
    (r^as^b, r^cs^d)&  \longmapsto\begin{cases}
         -1& \text{ if } b=d=1 \\
         1 & \text{ other }  \\
    \end{cases}
\end{align*}

Let $f:\mathbb{F}_q^{\alpha_1}D_{2n} \longrightarrow \mathbb{F}_q^{\alpha_1}D_{2n}$ defined over $\overline{r},\overline{s}$ given by: 
\begin{align*}
    f(r) & = \sum_{i=0}^{n-1} \gamma _i \overline{r^i} + \delta_i \overline{r^is} \\
    f(s) & = \sum_{i=0}^{n-1} h _i \overline{r^i} + t_i \overline{r^is}
\end{align*}
The relations that an extension $f^*$ should satisfy are $f^*(r^n) = f^*(s^2) = f^* (rsrs)=0$. To do so, note that $\beta_{r^n}=\prod_{i=1}^{n-1} \alpha(r^i,r) = 1$ and

\begin{align*}
    f^*(r^n) & =  \sum_{k=0}^{n-1} \left( \prod_{i=0}^{k-1} \overline{r}\right) f(r) \left( \prod_{i=k+1}^{n-1} \overline{r}\right) \\
    & = \sum_{k=0}^{n-1} \overline{r}^{k} \left( \sum_{i=0}^{n-1} \gamma _i \overline{r^i} + \delta_i \overline{r^is} \right) \overline{r}^{n-k-1} \\
    & =  \sum_{k=0}^{n-1} \sum_{i=0}^{n-1} \gamma _i \overline{r^i} + \delta_i   \overline{r^{i+k-(n-k-1)}s} \\
    & =  \sum_{k=0}^{n-1} \sum_{i=0}^{n-1} \gamma _i \overline{r^i} + \delta_i  \overline{r^{i+2k+1}s} \\
\end{align*}
From where

\begin{align} \label{Relnimpar}
    n \gamma _i &= 0 \\
   \sum_{k=0}^{n-1} \sum_{i=0}^{n-1}\delta_i 
   \overline{r^{i+2k+1}s}  & = 0
\end{align}

The first condition implies that $p| n$ and $\gamma _i$ is free, or that $p \nmid n$ and $\gamma _i=0$. For the second relation in (\ref{Relnimpar}), lets associate each term corresponding to the same basis elements, $\overline{r}^\Omega$. If $i+2k+1= \Omega \bmod n$, then $k= \frac{\Omega-i-1}{2} $ for all $\Omega, i$, so the system transforms into

\begin{equation}
    \sum_{i=0}^{n-1} \delta_{i} = 0
\end{equation}

Now, we focus on $f^*(s^2)=0$. This is equivalent to  $-f(s) = \frac{1}{\alpha(s,s)} \overline{s} f(s) \overline{s}$, from where 

\begin{align*}
    -f(s) & = \frac{1}{\alpha(s,s)} \overline{s} f(s) \overline{s} \\
    & =   - \overline{s} f(s) \overline{s} \\
    & = - \overline{s} \left( \sum_{i=0}^{n-1} h_i \overline{r^i} + t_i \overline{r^is} \right) \overline{s} \\
    & =  \sum_{i=0}^{n-1} h_i \alpha(s,r^i)\alpha(r^{-i}s,s) \overline{r^{-i}} + t_i \alpha(s,r^is)\alpha(r^{-i},s)\overline{r^{-i}s}   \\
    & = - \sum_{i=0}^{n-1} h_i (-1) \overline{r^{-i}} + t_i (-1)\overline{r^{-i}s}   \\
    & = \sum_{i=0}^{n-1} h_i \overline{r^{-i}} + t_i \overline{r^{-i}s}  
\end{align*}
and we get
\begin{equation}
    - \sum_{i=0}^{n-1} h_i \overline{r^i} + t_i \overline{r^is} = \sum_{i=0}^{n-1} h_i  \overline{r^{-i}} + t_i \overline{r^{-i}s}  
\end{equation}
Finally, we have the relations.
\begin{align*}
    h_i  + h_{n-i} & = 0 \\
    t_i + t_{n-i} & = 0 \\
\end{align*}

Now, for $f^*(rsrs)=0$, this is equivalent to $-f^*(rs) =\frac{1}{\alpha(rs,rs)} \overline{rs} f^*(rs) \overline{rs} $. Let 
\begin{equation}
    f^*(rs) = \sum_{i=0}^{n-1} w_i \overline{r^i} + A_i \overline{r^is}
\end{equation}
Then
\begin{align*}
    - f^*(rs) & = \frac{1}{\alpha(rs,rs)} \overline{rs} f^*(rs) \overline{rs} \\
    & = - \overline{rs} \left(\sum_{i=0}^{n-1} w_i \overline{r^i} + A_i \overline{r^is} \right) \overline{rs} \\
    & = - \sum_{i=0}^{n-1} w_i (-1) \overline{r^{-i}} + A_i (-1)\overline{r^{2-i}s} \\
    & =  \sum_{i=0}^{n-1} w_i \overline{r^{-i}} + A_i \overline{r^{2-i}s}
\end{align*}

And therefore 

\begin{equation}
    - \sum_{i=0}^{n-1} w_i \overline{r^i} + A_i \overline{r^is} =  \sum_{i=0}^{n-1} w_i \overline{r^{-i}} + A_i \overline{r^{2-i}s}
\end{equation}
finally, we end up with
\begin{align*}
    w_i  + w_{n-i} & = 0 \\
    A_i  + A_{2-i} & = 0
\end{align*}

Now, we have to compute $w_i,A_i$ in terms of $f(r),f(s)$
\begin{align*}
    f(rs) & = f\left(\frac{1}{\alpha(r,s)} r \cdot s \right) \\ 
    & = f(r \cdot s) \\
    & = f(r) \overline{s} + \overline{r} f(s) \\
    & = \left(\sum_{i=0}^{n-1} \gamma _i \overline{r^i} + \delta_i \overline{r^is} \right) \overline{s} + \overline{r} \left( \sum_{i=0}^{n-1} h _i \overline{r^i} + t_i \overline{r^is} \right) \\
    & = \sum_{i=0}^{n-1} \gamma _i \overline{r^is} + \delta_i \overline{r^i} + \sum_{i=0}^{n-1} h _i \overline{r^{i+1}} + t_i \overline{r^{i+1}s}
\end{align*} 
and therefore
\begin{align*}
    w_i = & \delta_i + h_{i-1} \\
    A_i = & \gamma_i + t_{i-1}
\end{align*}

And as a result we have the relations

\begin{align*}
    \delta_i + h_{i-1} + \delta_{-i} + h_{-i-1} & = 0 \\
    \gamma_i + t_{i-1} + \gamma_{2-i} + t_{1-i} & = 0
\end{align*}
As we know $t_{i-1} + t_{1-i}=0$, we have that this simplify to
\begin{align*}
       \delta_i + h_{i-1} + \delta_{-i} + h_{-i-1} & = 0 \\
    \gamma_i + \gamma_{2-i} & = 0
\end{align*}
Lets focus on the relations.
\begin{align*}
            \delta_i + h_{i-1} + \delta_{-i} + h_{-i-1} & = 0 \\
      h_i  + h_{n-i} & = 0 \\
\end{align*}
From here, we have that
\begin{align*}
           \delta_{-i}  & = - h_{-i-1} - h_{i-1} - \delta_i\\
      h_i  + h_{n-i} & = 0 \\
\end{align*}
We then have that $h_i,\delta_i$ for $i=1,\cdots, \frac{n-1}{2}$ determines the rest of $h_i,\delta_i$, and clearly they are independent of each other. Also, for $i=0$ we have $\delta_0 = - h_{1}$ as the characteristic is not $2$. Therefore, we have $2\frac{n-1}{2} = n-1$ degrees of freedom.

If we add the first equality for $i=1,2,\cdots, \frac{n-1}{2}$ we have 
\begin{align*}
0 = &
   \sum_{i=1}^{\frac{n-1}{2}} \delta_i + \sum_{i=1}^{\frac{n-1}{2}} \delta_{n-i} + \sum_{i=1}^{\frac{n-1}{2}} h_{i-1} + \sum_{i=1}^{\frac{n-1}{2}} h_{-i-1}  \\
   = &  \sum_{i=1}^{\frac{n-1}{2}} \delta_i + \sum_{i=\frac{n+1}{2}}^{n-1} \delta_{i} + \sum_{i=0}^{\frac{n-3}{2}} h_{i} + \sum_{i=\frac{n-1}{2}}^{n-2} h_{i} \\
    = &  \sum_{i=1}^{n-1} \delta_i + \sum_{i=0}^{n-2} h_i\\
     = &  \sum_{i=1}^{n-1} \delta_i  + h_{1}  \\
= &  \sum_{i=1}^{n-1} \delta_i   + \delta_{0}  \\
\end{align*}
As we wanted to prove.

If we want to study the dimension of this space, note that 
\begin{equation}
    \gamma_i + \gamma_{2-i} =0
\end{equation}
imply that $\gamma_1=0$, and that $\gamma_i = -\gamma_{2-i}$ for the rest of $i$, and as a result $\frac{n-1}{2}$ of $\gamma_i$ are free, so $ \frac{n-1}{2}$ of $\gamma_i$ are free. We can work similarly with $t_i + t_{n-i}$ but in this case $t_0=0$, so again $\frac{n-1}{2}$ of $t_i$ are free.

Then, we see that $Der(\mathbb{F}_q^{\alpha_1}D_{2n})$ is of dimension $(4 \frac{n-1}{2})= 2n-2$
\begin{proposition} \label{DerivDnimpar}
    Let $D_{2n}$ with $n$ odd, $q=p^m$ with $p$ an odd prime, and let $\mathbb{F}_q^{\alpha_3}D_{2n}$. Let $f:\mathbb{F}_q^{\alpha_3}D_{2n} \longrightarrow \mathbb{F}_q^{\alpha_3}D_{2n}$ a derivation. Then defined over $\overline{r},\overline{s}$ given by: 
\begin{align*}
    f(r) & = \sum_{i=0}^{n-1} \gamma _i \overline{r^i} + \delta_i \overline{r^is} \\
    f(s) & = \sum_{i=0}^{n-1} h _i \overline{r^i} + t_i \overline{r^is}
\end{align*}
with
\begin{align*}
     \delta_i + h_{i-1} + \delta_{-i} + h_{-i-1} & = 0 \\
    \gamma_i + \gamma_{2-i} & = 0 \\
        h_i  + h_{n-i} & = 0 \\
    t_i + t_{n-i} & = 0 \\
       n \gamma_i =0
\end{align*}
and extended by the Leibtnitz rule.  In addition, in case $p|n$, then $Der(\mathbb{F}_q^{\alpha_3}D_{2n})$ is a $\mathbb{F}_q-$vector space of dimension $2n-2$
\end{proposition}

\section{Application to computation of $HH^1(\mathbb{F}_q^\alpha D_{2n})$}
We will present how our results can be used to calculate $HH^1$ of the twisted group ring. To do so, we present how it can be done in case $D_{2n}$ with $n$ odd, and as an example in case $n$ is even, we will deal with the case $n=2n'$ with $n'$ odd.
\subsection{$HH^1(\mathbb{F}_q^{\alpha_3} D_{2n})$ with $n$ odd}

We know that $p\nmid n$, then $HH^1(\mathbb{F}_q^{\alpha_3} D_{2n})=0$, and as a result all derivations are inner derivations. Now, in the case $p|n$, we will prove that we have dimension $\frac{n-1}{2}$
 
The inner derivations are the images of the following homomorphism of $\mathbb{F}_q-$vector spaces. 
\begin{align*}
    \varphi: \mathbb{F}_q^\alpha D_{2n} & \longrightarrow Der(\mathbb{F}_q^\alpha D_{2n}) \\
    a & \longmapsto (x\longmapsto ax-xa)
\end{align*}
as $ker(\varphi) =Z(\mathbb{F}_qD_{2n})$, we have that
\begin{equation}
    \frac{\mathbb{F}_q^\alpha D_{2n}}{Z(\mathbb{F}_qD_{2n})} \cong Inn(\mathbb{F}_qD_{2n})
\end{equation}
 
 As $\mathbb{F}_qD_{2n}$ is a $\mathbb{F}_q$ algebra of dimension $2n$, then $Inn(\mathbb{F}_qD_{2n})$ is a $\mathbb{F}_q-$vector space with dimension $2n-dim_{\mathbb{F}_q}(Z(\mathbb{F}_qD_{2n}))$.

The center of a twisted group ring is characterized by
\begin{proposition}
Let $R$ a ring, $G$ a group and $\alpha: G\times G \longrightarrow R$ a $2-$cocycle. Then
    \[
Z(R^\alpha[G]) =
\Bigg\{
\sum_{g \in G} a_g \, \overline{g} \;\Bigg|\; 
a_{h g h^{-1}} \, \alpha( hg h^{-1} , h) = a_g \, \alpha(h,g) 
\;\; \forall g,h \in G
\Bigg\}.
\]
\end{proposition}
In our case, we have the relations 
\begin{align*}
    a_{r^j} \alpha(r^is,r^j) = &  a_{r^{-j}}\alpha(r^{-j},r^is)\\
    a_{r^js} \alpha(r^i,r^js) = & a_{r^{2i+j}s} \alpha(r^{2i+j}s, r^i)\\
    a_{r^js} \alpha(r^is,r^js) = & a_{r^{2j-i}s} \alpha( r^{2i-j}s, r^is)
\end{align*}
This turns out to be
\begin{align*}
    a_{r^j}  = &  a_{r^{-j}}\\
    a_{r^js} = & a_{r^{2i+j}s} \\
    - a_{r^js} = & - a_{r^{2i-j}s} 
\end{align*}
as $n$ is odd, then $2$ is invertible and $2j+i= k \mod{n}$ has a solution for all $i,j,k$, so this relations implies $a_{r^is}=c\in \mathbb{F}_q$ constant. The relation $a_{r^i} = a_{r^{-i}}$ implies that $a_{id}$ is free and half of the other are free. This implies that the dimension as $\mathbb{F}_q-$vector space of the center is $1 + \frac{n-1}{2}+ 1= \frac{n+3}{2}$. Finally, the dimension of $Inn(\mathbb{F}_{p^m}^\alpha D_{2n})$ is $2n-\frac{n+3}{2} = \frac{3n-3}{2}$, while $dim ( Der(\mathbb{F}_{p^m}^\alpha D_{2n}) / Inn(\mathbb{F}_{p^m}^\alpha D_{2n}) = 2n -2 - \frac{3n-3}{2} = \frac{n-1}{2}$

\begin{proposition}
      Let $D_{2n}$ with $n$ odd, $q=p^m$ with $p$ a odd prime, and let $\mathbb{F}_q^{\alpha_1}D_{2n}$.Let $\alpha_3$ be the $2-$cocycles given by

\begin{align*}
    \alpha_3: D_{2m} \times D_{2m} & \longrightarrow \mathbb{F}_{q}^* \\
    (r^as^b, r^cs^d)&  \longmapsto\begin{cases}
         -1 & \text{ if } b=d=1 \\
         1 & \text{ other }  \\
    \end{cases}
\end{align*}
Then $HH^1(\mathbb{F}_qD_{2n}) \cong \mathbb{F}_q^{\frac{n-1}{2}}$
\end{proposition}

\subsection{$HH^1(\mathbb{F}_q^{\alpha_1} D_{2n})$ with $n=2n'$ with $n'$ odd}
As in the previous case, we have to study the center of the ring, given by

\begin{align*}
    a_{r^j} \alpha(r^is,r^j) = &  a_{r^{-j}}\alpha(r^{-j},r^is)\\
    a_{r^js} \alpha(r^i,r^js) = & a_{r^{2i+j}s} \alpha(r^{2i+j}s, r^i)\\
    a_{r^js} \alpha(r^is,r^js) = & a_{r^{2j-i}s} \alpha( r^{2i-j}s, r^is)
\end{align*}
This turns out to be
\begin{align*}
    (-1)^ja_{r^j}  = &  a_{r^{-j}}\\
    a_{r^js} = & (-1)^i a_{r^{2i+j}s} \\
    (-1)^j a_{r^js} = & (-1)^i a_{r^{2i-j}s} 
\end{align*}

In the first equation, we get that $a_0$ is free, $a_{n'}$ is $0$ and $\frac{n-2}{2}$ of the $a_{r^i}$ are free. In the second equation, we can take $i=n'$ and we get $a_{r^js} = 0$. Therefore, the center is of dimension $\frac{n}{2}=n'$ and as in the previous case, we have that $dim(HH^1(\mathbb{F}_q^{\alpha_1} D_{2n}) ) = 2n-\frac{n}{2} = \frac{3n}{2}= 3n'$
\begin{corollary}
    Let $D_{2n}$ with $n=2n'$, $n'$ odd, $q=p^m$ with $p|n$ odd prime and let $\mathbb{F}_q^{\alpha_1} D_{2n}$ with
    \begin{align*}
    \alpha_1: D_{2m} \times D_{2m} & \longrightarrow \mathbb{F}_{q}^* \\
    (r^as^b, r^cs^d)& \longmapsto(-1)^{bc}
\end{align*}
Then $HH^1(\mathbb{F}_q^{\alpha_1} D_{2n}) = \mathbb{F}_q^{3n'}$
\end{corollary}

\section{Examples}
In this section we will present several examples of the different results that appear in this paper.

\begin{example}
        Let $D_{12}$ be the dihedral group of $12$ elements, and $\mathbb{F}_{9}$ the finite field of $9$ elements. Let 
\begin{align*}
    \alpha_1: D_{12} \times D_{12} & \longrightarrow \mathbb{F}_{9}^* \\
    (r^as^b, r^cs^d)& \longmapsto(-1)^{bc}
\end{align*}

be a normalized $2-$cocycle. Then linear equations of Proposition \ref{DerivParn'impar} can be expressed as the kernel of 
\begin{equation}
\scalebox{0.6}{$
\left( \begin{array}{cccccccccccccccccccccccc}
2&0&0&0&0&0&0&0&0&0&0&2&0&0&0&0&0&0&0&0&0&0&0&0\\
0&-1&0&0&0&1&-1&0&0&0&1&0&0&0&0&0&0&0&0&0&0&0&0&0\\
0&0&1&0&1&0&0&1&0&1&0&0&0&0&0&0&0&0&0&0&0&0&0&0\\
0&0&0&0&0&0&0&0&0&0&0&0&0&0&0&0&0&0&0&0&0&0&0&0\\
0&0&1&0&1&0&0&1&0&1&0&0&0&0&0&0&0&0&0&0&0&0&0&0\\
0&1&0&0&0&-1&1&0&0&0&-1&0&0&0&0&0&0&0&0&0&0&0&0&0\\
0&0&0&0&0&0&0&0&0&0&0&0&1&0&1&0&0&0&0&0&0&0&0&0\\
0&0&0&0&0&0&0&0&0&0&0&0&0&0&0&0&0&0&0&0&0&0&0&0\\
0&0&0&0&0&0&0&0&0&0&0&0&1&0&1&0&0&0&0&0&0&0&0&0\\
0&0&0&0&0&0&0&0&0&0&0&0&0&0&0&-1&0&1&0&0&0&0&0&0\\
0&0&0&0&0&0&0&0&0&0&0&0&0&0&0&0&2&0&0&0&0&0&0&0\\
0&0&0&0&0&0&0&0&0&0&0&0&0&0&0&1&0&-1&0&0&0&0&0&0\\
0&0&0&0&0&0&2&0&0&0&0&0&0&0&0&0&0&0&0&0&0&0&0&0\\
0&0&0&0&0&0&0&-1&0&0&0&1&0&0&0&0&0&0&0&0&0&0&0&0\\
0&0&0&0&0&0&0&0&1&0&1&0&0&0&0&0&0&0&0&0&0&0&0&0\\
0&0&0&0&0&0&0&0&0&0&0&0&0&0&0&0&0&0&0&0&0&0&0&0\\
0&0&0&0&0&0&0&0&1&0&1&0&0&0&0&0&0&0&0&0&0&0&0&0\\
0&0&0&0&0&0&0&1&0&0&0&-1&0&0&0&0&0&0&0&0&0&0&0&0\\
0&0&0&0&0&0&0&0&0&0&0&0&0&0&0&0&0&0&0&0&0&0&0&0\\
0&0&0&0&0&0&0&0&0&0&0&0&0&0&0&0&0&0&0&1&0&0&0&1\\
0&0&0&0&0&0&0&0&0&0&0&0&0&0&0&0&0&0&0&0&-1&0&1&0\\
0&0&0&0&0&0&0&0&0&0&0&0&0&0&0&0&0&0&0&0&0&2&0&0\\
0&0&0&0&0&0&0&0&0&0&0&0&0&0&0&0&0&0&0&0&1&0&-1&0\\
0&0&0&0&0&0&0&0&0&0&0&0&0&0&0&0&0&0&0&1&0&0&0&1 
\end{array} \right)$}
\end{equation}
and with the study of the previous section, the derivations are given by

\begin{center}
{%
\renewcommand{\arraystretch}{1.5} 
\begin{tabular}{c|c|c}
 & $f_i(r)$ & $f_i(s)$ \\
\hline
$f_1$ & $\overline{r^3s}$ & $0$ \\
$f_2$ & $- \overline{r^2s} + \overline{r^4s} $ & $0$ \\
$f_3$ & $\overline{rs} + \overline{r^5s} $ & $0$ \\
$f_4$ & $ - \overline{r^2s} $ & $\overline{r^3} $ \\
$f_5$ & $ \overline{rs} $ & $\overline{r^2} +\overline{r^4} $ \\
$f_6$ & $ -\overline{s}-\overline{r^2s} $ & $\overline{r} +\overline{r^5} $ \\
$f_7$ & $ \overline{r}$ & $0$ \\
$f_8$ & $ - \overline{e} + \overline{r^2}$ & $0$ \\
$f_9$ & $ - \overline{r^3} + \overline{r^5}$ & $0$ \\
$f_{10}$ & $ 0$ & $\overline{s}$ \\
$f_{11}$ & $ 0$ & $\overline{r^2s} + \overline{r^4s}$ \\
$f_{12}$ & $ 0$ & $\overline{rs} + \overline{r^5s}$ \\
\end{tabular}
}%
\end{center}
\end{example}

\begin{example}
        Let $D_{6}$ be the dihedral group of $6$ elements and $\mathbb{F}_{9}$ be the finite field of $9$ elements. Let 
    \begin{align*}
    \alpha_3: D_{6} \times D_{6} & \longrightarrow \mathbb{F}_{9}^* \\
    (r^as^b, r^cs^d)&  \longmapsto\begin{cases}
         -1& \text{ if } b=d=1 \\
         1 & \text{ other }  \\
    \end{cases}
\end{align*}
be a normalized $2-$cocycle. Then linear equations of Proposition \ref{DerivDnimpar} can be expressed as the kernel of 
\begin{equation}
\left( \begin{array}{ccccccccccccc}
1&1&1&0&0&0&0&0&0&0&0&0\\
2&0&0&0&0&2&0&0&0&0&0&0\\
0&1&1&1&1&0&0&0&0&0&0&0\\
0&1&1&1&1&0&0&0&0&0&0&0\\
0&0&0&0&0&0&1&0&1&0&0&0\\
0&0&0&0&0&0&0&2&0&0&0&0\\
0&0&0&0&0&0&1&0&1&0&0&0\\
0&0&0&2&0&0&0&0&0&0&0&0\\
0&0&0&0&1&1&0&0&0&0&0&0\\
0&0&0&0&1&1&0&0&0&0&0&0\\
0&0&0&0&0&0&0&0&0&2&0&0\\
0&0&0&0&0&0&0&0&0&0&1&1\\
0&0&0&0&0&0&0&0&0&0&1&1 
\end{array} \right)
\end{equation}
and with the study of the previous section, the kernel is the set
\begin{equation}
\left\{
\left(
   \begin{array}{c}
     -b \\
     -a +b \\
     a\\
     0 \\
     -b \\
     b \\ 
    -c \\
     0 \\
     c\\
     0 \\
     -d \\
     d \\ 
\end{array} 
\right);
a,b,c\in \mathbb{F}_{9}
\right\}
\end{equation}
As a result, $Der(\mathbb{F}_{9}^\alpha D_{6})$ are linear combinations of
\begin{center}
{%
\renewcommand{\arraystretch}{1.5} 
\begin{tabular}{c|c|c}
 & $f_i(r)$ & $f_i(s)$ \\
\hline
$f_1$ & $- \overline{rs} + \overline{r^2s}$ & $0$ \\
$f_2$ & $ -\overline{s} + \overline{rs}$ & $- \overline{r} + \overline{r^2}$ \\
$f_3$ & $ -\overline{e} + \overline{r^2}$ & $0$ \\
$f_4$ & $ 0 $ & $ -\overline{rs} + \overline{r^2s}$ 
\end{tabular}
}%
\end{center}
Also, we have that the dimension of the inner derivations in this case is $3$. In particular,  we have that
\begin{align*}
    f_1(x) & = \overline{s}x-x\overline{s} \\
    f_2(x) & = x \overline{rs} - \overline{rs} x \\
    f_4(x) & = x \overline{r} - \overline{r}x 
\end{align*}
So $f_1,f_2,f_3 \in Inn(\mathbb{F}_{9}^\alpha D_{10})$ and $HH^1(\mathbb{F}_{9}^\alpha D_{10}) = <f_3>$ as $\mathbb{F}_{9}$ vector space. 
\end{example}

\bibliographystyle{plain}
\bibliography{ArticuloIngles}

\begin{thebibliography}{10}

\bibitem{Brauer1929}
Richard Brauer.
\newblock \"uber systeme hyperkomplexer zahlen.
\newblock {\em Mathematische Zeitschrift}, 30:79--107, 1929.

\bibitem{zbMATH03935317}
Kenneth~S. Brown.
\newblock {\em Cohomology of groups}, volume~87 of {\em Grad. Texts Math.}
\newblock Springer, Cham, 1982.

\bibitem{CREEDON2019247}
Leo Creedon and Kieran Hughes.
\newblock Derivations on group algebras with coding theory applications.
\newblock {\em Finite Fields and Their Applications}, 56:247--265, 2019.

\bibitem{DelaCruz2021}
Javier De~La~Cruz and Wolfgang Willems.
\newblock Twisted group codes.
\newblock {\em IEEE Transactions on Information Theory}, 67(8):5178--5184, 2021.

\bibitem{Ferrero95}
M.~Ferrero, A.~Giambruno, and C.~Polcino Milies.
\newblock A note on derivations of group rings.
\newblock {\em Canadian Mathematical Bulletin}, 38(4):434, 1995.

\bibitem{Fleischmann1993}
Peter Fleischmann, Ingo Janiszczak, and Wolfgang Lempken.
\newblock Finite groups have local non-{Schur} centralizers.
\newblock {\em Manuscripta Mathematica}, 80:213--224, 1993.

\bibitem{Frobenius1896b}
Ferdinand~Georg Frobenius.
\newblock {\"U}ber die primfactoren der gruppendeterminante.
\newblock {\em Sitzungsberichte der K{\"o}niglich Preu{\ss}ischen Akademie der Wissenschaften zu Berlin}, pages 1343--1382, 1896.

\bibitem{Frobenius1896}
Ferdinand~Georg Frobenius.
\newblock {\"U}ber gruppencharaktere.
\newblock {\em Sitzungsberichte der K{\"o}niglich Preu{\ss}ischen Akademie der Wissenschaften zu Berlin}, pages 985--1021, 1896.

\bibitem{HandelmanLawrenceSchelter1978}
David Handelman, John Lawrence, and William Schelter.
\newblock Skew group rings.
\newblock {\em Houston Journal of Mathematics}, 4(2):175--198, 1978.

\bibitem{Herstein1968}
I.~N. Herstein.
\newblock {\em Noncommutative Rings}, volume~15 of {\em Carus Mathematical Monographs}.
\newblock The Mathematical Association of America, Washington, D.C., 1968.

\bibitem{Hochschild1945}
G.~Hochschild.
\newblock On the cohomology groups of an associative algebra.
\newblock {\em Annals of Mathematics}, 46(1):58--67, 1945.

\bibitem{huang2020explicit}
H.-L. Huang, Z.~Wan, and Y.~Ye.
\newblock Explicit cocycle formulas on finite abelian groups with applications to braided linear gr-categories and dijkgraaf--witten invariants.
\newblock {\em Proceedings of the Royal Society of Edinburgh: Section A Mathematics}, 150(4):1937--1964, 2020.

\bibitem{Hughes2001}
G.~Hughes.
\newblock Structure theorems for group ring codes with an application to self-dual codes.
\newblock {\em Designs, Codes and Cryptography}, 24(1):5--14, 2001.

\bibitem{HurleyHurley2009}
P.~Hurley and T.~Hurley.
\newblock Codes from zero-divisors and units in group rings.
\newblock {\em International Journal of Information and Coding Theory}, 1(1):57--87, 2009.

\bibitem{karpilovsky1987algebraic}
G.~Karpilovsky.
\newblock {\em The Algebraic Structure of Crossed Products}, volume 142 of {\em North-Holland Mathematics Studies}.
\newblock Elsevier Science, 1987.

\bibitem{Maschke1899}
Hugo Maschke.
\newblock Beweis des satzes, da{\ss} eine endliche gruppe einer linearen substitutionengruppe gleichwertig ist.
\newblock {\em Mathematische Annalen}, 52:363--368, 1899.

\bibitem{McConnell1975}
J.~C. McConnell.
\newblock Representations of solvable lie algebras. ii. twisted group rings.
\newblock {\em Annales Scientifiques de l'{\'E}cole Normale Sup{\'e}rieure, 4e s{\'e}rie}, 8(2):157--178, 1975.

\bibitem{Molien1897}
Theodor Molien.
\newblock {\"U}ber die invarianten der linearen substitutionsgruppen.
\newblock {\em Sitzungsberichte der K{\"o}niglich Preu{\ss}ischen Akademie der Wissenschaften zu Berlin}, pages 1152--1156, 1897.

\bibitem{Noether1929}
Emmy Noether.
\newblock Hyperkomplexe gr{\"o}{\ss}ensysteme und ihre beziehungen zur kommutativen algebra und zahlentheorie.
\newblock {\em Mathematische Zeitschrift}, 30:641--692, 1929.

\bibitem{Passman1989}
Donald~S. Passman.
\newblock {\em Infinite Crossed Products}.
\newblock Academic Press, Boston, 1989.

\bibitem{RossoSavage2017}
Daniele Rosso and Alistair Savage.
\newblock A general approach to heisenberg categorification via wreath product algebras.
\newblock {\em Mathematische Zeitschrift}, 286(1-2):603--655, 2017.

\bibitem{Schur1905}
Issai Schur.
\newblock Neue begr{\"u}ndung der theorie der gruppencharaktere.
\newblock {\em Sitzungsberichte der K{\"o}niglich Preu{\ss}ischen Akademie der Wissenschaften zu Berlin}, pages 406--432, 1905.

\bibitem{TODEA2023107192}
Constantin-Cosmin Todea.
\newblock Nontriviality of the first hochschild cohomology of some block algebras of finite groups.
\newblock {\em Journal of Pure and Applied Algebra}, 227(2):107192, 2023.

\end{thebibliography}
\end{document}